\documentclass{article}

\usepackage[active]{srcltx}

\usepackage{graphicx}

\usepackage{latexsym}
\usepackage{amsfonts}
\usepackage[psamsfonts]{amssymb}
\usepackage{enumerate}

\RequirePackage[OT1]{fontenc}
\RequirePackage[amsthm,amsmath]{imsart}

\theoremstyle{plain} 
\newtheorem{theorem}{Theorem}

\theoremstyle{definition} 

\theoremstyle{definition} 

\theoremstyle{remark} 

\theoremstyle{remark} 

\newtheorem*{remark*}{Remark}

\newcommand{\eqD}{\stackrel{\mathrm{D}}=}

\renewcommand{\le}{\leqslant}
\renewcommand{\ge}{\geqslant}

\newcommand{\si}{\sigma}

\renewcommand{\Psi}{\overline{\Phi}}

\newcommand{\E}{\operatorname{\mathsf{E}}}
\newcommand{\Var}{\operatorname{\mathsf{Var}}}
\newcommand{\Cov}{\operatorname{\mathsf{Cov}}}

\begin{document}


\begin{frontmatter}

\title{Relations between the first four moments}
\runtitle{First four moments}

\begin{aug}
\author{\fnms{Iosif} \snm{Pinelis}\ead[label=e1]{ipinelis@math.mtu.edu}}
\runauthor{Iosif Pinelis}


\address{Department of Mathematical Sciences\\
Michigan Technological University\\
Houghton, Michigan 49931, USA\\
E-mail: \printead[ipinelis@mtu.edu]{e1}}
\end{aug}

\begin{abstract}
One of the results is that $\E X^3\le(\frac 4{27})^{1/4}(\E X^4)^{3/4}$ for all random variables $X$ with $\E X\le0$, and the constant factor $(\frac 4{27})^{1/4}$ here is the best possible. 
\end{abstract}







\end{frontmatter}



Let $X$ be any random variable (r.v.) with moments 
\begin{equation*}
	m_j:=\E X^j
\end{equation*}
for $j=0,1,\dots$, using the convention $0^0:=1$, so that $m_0=1$. 

It is clear that $m_3\le\E|X|^3\le m_4^{3/4}$. Moreover, $m_3=m_4^{3/4}$ if (and only if) the r.v.\ $X$ is a nonnegative constant. So, $c=1$ is the best constant factor in the inequality 
\begin{equation}\label{eq:c}
	m_3\le cm_4^{3/4}
\end{equation}
over all r.v.'s $X$ satisfying the condition $0<m_4<\infty$, which will be henceforth assumed. 

However, it will shown in this note that the constant $c$ in \eqref{eq:c} can be improved precisely to $(\frac 4{27})^{1/4}=0.620\dots$ under the additional condition 
\begin{equation}\label{eq:conds}
m_1\le0, 
\end{equation}
which will be henceforth assumed as well. 
Condition \eqref{eq:conds} is satisfied in many applications, when the r.v.\ $X$ is either assumed to be zero-mean or is obtained by truncating a zero-mean r.v.\ from above. 

For any positive real $u$ and $v$, let $X_{u,v}$ stand for any zero-mean r.v.\ with values in the set $\{-u,v\}$; note that, given any such $u$ and $v$, the distribution of $X_{u,v}$ is uniquely determined.  


\begin{theorem}\label{th:1} 
One has 
\begin{align}
	m_3&\le\sqrt{m_4m_2-m_2^3} \label{eq:1} \\
	\intertext{and hence} 
	m_3&\le\Big(\frac 4{27}\Big)^{1/4}
	\,m_4^{3/4}.  \label{eq:2}
\end{align}
The equality in \eqref{eq:1} obtains if and only if $X\eqD X_{u,v}$ for some positive real $u$ and $v$, where $\eqD$ denotes the equality in distribution. 
The equality in \eqref{eq:2} obtains if and only if 
$X\eqD X_{u,v}$ with $u=\frac{\sqrt3-1}{\sqrt2}\,\si$ and $v=\frac{\sqrt3+1}{\sqrt2}\,\si$ for some $\si\in(0,\infty)$. 
\end{theorem}
Note 
that the expression $m_4m_2-m_2^3$ under the square root in \eqref{eq:1} is always nonnegative. 

\begin{proof}[Proof of Theorem~\ref{th:1}] 
First here, it is straightforward to check the ``if'' parts of the statements about the equalities in \eqref{eq:1} and \eqref{eq:2}. 

Next, note that if $m_3<0$ then inequalities \eqref{eq:1} and \eqref{eq:2} are trivial. 
So, let us 
assume that $m_3\ge0$. 
The determinant of 
the obviously nonnegative quadratic form 
$ 
	Q(a_0,a_1,a_2):=\E(a_0+a_1X+a_2X^2)^2=\sum_{i,j=0}^2m_{i+j}a_ia_j  
$ 
is nonnegative. 
Therefore and by \eqref{eq:conds}, 
\begin{align}
	m_3^2&\le m_4m_2-m_2^3-m_1^2m_4+2m_1m_2m_3 \label{eq:ineq1} \\ 
	&\le m_4m_2-m_2^3, \label{eq:ineq2}
\end{align}
which implies inequality \eqref{eq:1}.   

Further, the equality in \eqref{eq:1} obtains only if both inequalities in \eqref{eq:ineq1} and \eqref{eq:ineq2} are in fact equalities. The equality in \eqref{eq:ineq2} implies that $m_1^2m_4=0$ and hence $m_1=0$. 
The equality in \eqref{eq:ineq1}
means that the determinant of the quadratic form $Q$ is zero or, equivalently, $a_0+a_1X+a_2X^2=0$ almost surely for some real $a_0,a_1,a_2$ such that at least one of them is nonzero; in other words, the support of the distribution of $X$ consists of at most two points (the real roots of the quadratic polynomial $a_0+a_1x+a_2x^2$), and this distribution is zero-mean. 
Thus, the necessary and sufficient condition for the equality in \eqref{eq:1} is verified. 

The upper bound in \eqref{eq:2} is obtained by the maximization in $m_2$ of the upper bound in \eqref{eq:1}, with the maximizer $m_2=\sqrt{m_4/3}$. 
Accordingly, the necessary and sufficient condition for the equality in \eqref{eq:2} follows from that for the equality in \eqref{eq:1}; at that, $\si^2=m_2=\sqrt{m_4/3}$. 
\end{proof}

\begin{remark*}
Inequality \eqref{eq:ineq1} can be rewritten as $\Cov(X^2,X)^2\le\Var(X^2)\Var X$, which is an instance of the Cauchy-Schwarz inequality. 
Also, one can use \eqref{eq:ineq1} to obtain exact upper and lower bounds on $m_3$ under conditions other than \eqref{eq:conds}. 
\end{remark*}

\end{document}